\documentclass[12pt]{article}
\usepackage{mathtools}

\usepackage{xcolor}
\definecolor{mycolor1}{rgb}{0.00000,0.44700,0.74100}
\definecolor{mycolor2}{rgb}{0.8500, 0.3250, 0.0980}
\definecolor{mycolor3}{rgb}{0.9290, 0.6940, 0.1250}
\definecolor{mycolor4}{rgb}{0.4940, 0.1840, 0.5560}

\usepackage{algorithm}
\usepackage{mathtools}

\usepackage{stmaryrd}
\usepackage[margin=1.04in]{geometry}
\usepackage{amsthm}
\theoremstyle{plain}
\newtheorem{theorem}{Theorem}[section]
\newtheorem{lemma}{Lemma}[section]
\newtheorem{proposition}{Proposition}[section]

\theoremstyle{definition}
\newtheorem{definition}{Definition}[section]

\usepackage{multicol}

\newtheorem{remark}{Remark}

\newcommand{\NBE}{\textup{NBE}}
\newcommand{\EBE}{\textup{EBE}}
\newcommand{\EMBE}{\textup{EMBE}}
\newcommand{\TBE}{\textup{TBE}}

\usepackage{url}

\newcommand{\C}{\mathbb{C}}

\newcommand{\R}{\mathbb{R}}
\usepackage{tikz}
\usepackage{pgfplots}
\usepackage{tikz-3dplot}
\usepackage{mathrsfs}

\usepackage{blkarray}
\usepackage{tikz-cd}
\usepackage{verbatim}
\usepackage{multirow}
\usepackage{authblk}
\usetikzlibrary{external}
\pgfplotsset{
  log x ticks with fixed point/.style={
      xticklabel={
        \pgfkeys{/pgf/fpu=true}
        \pgfmathparse{exp(\tick)}%
        \pgfmathprintnumber[fixed relative, precision=3]{\pgfmathresult}
        \pgfkeys{/pgf/fpu=false}
      }
  }}
\usepackage{array}
\usepackage{enumitem}
\usepackage{authblk}
\usepackage{blindtext}

\usepackage{adjustbox}
\newcolumntype{R}[2]{%
    >{\adjustbox{angle=#1,lap=\width-(#2)}\bgroup}%
    l%
    <{\egroup}%
}
\makeatother
\usepackage{makecell}
\usepackage{amssymb}
\usepackage{subfig}
\makeatletter
\newcommand{\xdashrightarrow}[2][]{\ext@arrow 0359\rightarrowfill@@{#1}{#2}}
\makeatother

\newcommand{\Binom}[2]{\begin{pmatrix}
{#1} \\ {#2}
\end{pmatrix}}
\newcommand{\norm}[1]{\lVert {#1} \rVert}

\newcommand{\trop}[1]{\texttt{t}{#1}}

\definecolor{sascha_color}{RGB}{0,139,139}

\begin{document}
\title{Backward Error Measures for Roots of Polynomials}
\author[1]{Simon Telen}
\author[2]{Sascha Timme\thanks{This author was supported by the Deutsche Forschungsgemeinschaft (German Research Foundation) Graduiertenkolleg {\em Facets of Complexity} (GRK~2434).}}
\author[1]{Marc Van Barel\thanks{This author was partially supported by the Research Council KU Leuven, C1-project (Numerical Linear Algebra and Polynomial Computations), and by the Fund for Scientific Research Flanders (Belgium), G.0828.14N (Multivariate polynomial and rational interpolation and approximation), and EOS Project no 30468160.}}
\affil[1]{\small{Department of Computer Science, KU Leuven}}
\affil[2]{\small{Technische Universit\"at Berlin, Chair of Discrete Mathematics/Geometry}}
\affil[ ]{\small{\texttt{simon.telen@kuleuven.be, timme@math.tu-berlin.de, marc.vanbarel@kuleuven.be}}}
\date{\today}

\maketitle

\begin{abstract}
We analyze different measures for the backward error of a set of numerical approximations for the roots of a polynomial. We focus mainly on the element-wise mixed backward error introduced by Mastronardi and Van Dooren, and the tropical backward error introduced by Tisseur and Van Barel. We show that these measures are equivalent under suitable assumptions. We also show relations between these measures and the classical element-wise and norm-wise backward error measures.
\end{abstract}

\section{Introduction}
In this article we analyze the problem of measuring the backward error for a set of approximations for the roots of a polynomial with complex coefficients. For a general introduction to the notion of backward error analysis, the reader can consult for instance \cite[Section 1.5]{Higham:2002}. Consider a set of approximate solutions $\hat{X} = \{ \hat{x}_1, \ldots, \hat{x}_d \} \subset \C^* = \C \setminus \{0\}$ of a polynomial equation with nonzero coefficients
$$f = c_0 + c_1x + \ldots +c_{d-1} x^{d-1} + c_d x^d = c_d(x-x_1) \cdots (x-x_d) = 0, \quad f \in \C[x],$$
with solutions $X = \{x_1, \ldots, x_d\} \subset \C^*$. The \emph{backward error} of $\hat{X}$ is a measure for the `distance' of $f$ to the polynomial $$\hat{f} = \hat{c}_0 + \hat{c}_1x + \ldots + \hat{c}_{d-1} x^{d-1} + c_d x^d = c_d(x- \hat{x}_1) \cdots (x - \hat{x}_d),$$ whose roots are exactly the points in $\hat{X}$. How to measure this distance turns out to be a surprisingly subtle problem. A first and natural measure is the 2-norm distance between the coefficients of $f$ and $\hat{f}$. The \emph{norm-wise backward error} (NBE) of $\hat{X}$ is 
$$ \NBE(\hat{X}) = \sqrt{\frac{|c_0-\hat{c}_0|^2 + \cdots + |c_{d-1}-\hat{c}_{d-1}|^2}{|c_0|^2 + \cdots + |c_{d}|^2}} = \frac{\norm{c - \hat{c}}_2}{\norm{c}_2},$$
where $c = (c_0, \ldots, c_{d}) \in \C^d, \hat{c} = (\hat{c}_0, \ldots, c_{d}) \in \C^d$. In \cite{aurentz2015fast}, an algorithm is proposed that computes a set of approximate solutions $\hat{X}$ satisfying $\NBE(\hat{X}) = O(u)$, where $u$ is the unit round-off. Such an algorithm is called \emph{norm-wise backward stable}. However, it turns out that this type of stability is too `weak' in a sense we explain by means of an example. Consider the polynomial $f = a(x-10^6)(x-10^{-6})$. The set of approximate solutions $\hat{X} = \{10^6 + u 10^6, 10^{-6} + u \}$ would satisfy $\NBE(\hat{X}) = O(u)$. Indeed, 
$$ \hat{c} = a(1 + (u 10^6 + u + u^2 10^6), -(10^6 + 10^{-6}) - u (10^6+1), 1), \quad c = a(1,-(10^6+10^{-6}),1)$$ 
and $\norm{c-\hat{c}}_2/\norm{c}_2 = O(u)$. This means that we would allow a relative error of size $u 10^6$ on the coefficient vector $c$. However, computing the roots of $f$ with $a = 0.2$ using the Julia package \texttt{PolynomialRoots} (using the command \texttt{roots}) we get
\begin{verbatim}
julia> abs.((c - chat)./c)
3-element Array{Float64,1}:
 2.7755575615628914e-16
 0.0                   
 0.0      
\end{verbatim}
which shows that we can obtain better element-wise accuracy on the coefficient vector. This suggests another, more `strict', measure for the backward error. The \emph{element-wise backward error} (EBE) of $\hat{X}$ is 
$$ \EBE(\hat{X}) = \max_{i=0, \ldots, d-1} \left | \frac{c_i - \hat{c}_i}{c_i} \right |.$$
Unfortunately, this measure turns out to be \emph{too} strict. In \cite{mastronardi2015revisiting}, Mastronardi and Van Dooren show that \emph{no algorithm} for finding the roots of a general quadratic polynomial is element-wise backward stable, meaning that it computes $\hat{X}$ such that $\EBE(\hat{X}) = O(u)$. As an alternative measure, the authors of \cite{mastronardi2015revisiting} propose the following definition in the case where $d = 2$. The \emph{element-wise mixed backward error} (EMBE) of $\hat{X}$, denoted $\EMBE(\hat{X})$, is the smallest number $\varepsilon \geq 0$ such that there exists some $\tilde{X} = \{ \tilde{x}_1, \ldots, \tilde{x}_d \} \subset \C$ and 
$$\tilde{f} = \tilde{c}_0 + \tilde{c}_1x + \cdots + \tilde{c}_{d-1}x^{d-1} + c_d x^d = c_d(x-\tilde{x}_1)\cdots(x-\tilde{x}_d)$$
such that 
\begin{align*}
|\hat{x}_i - \tilde{x}_i| &\leq \varepsilon |\tilde{x}_i|, \quad i = 1, \ldots, d,\\
|c_i - \tilde{c}_i| &\leq \varepsilon |c_i|, \quad i = 0, \ldots, d-1.
\end{align*}
Note that the second set of inequalities is equivalent to $\EBE(\tilde{X}) \leq \varepsilon$. In the same paper, the authors also show the implication $\EMBE(\hat{X}) = O(u) \Rightarrow \NBE(\hat{X}) = O(u)$ in the case where $d=2$. The implication $\EBE(\hat{X}) = O(u) \Rightarrow \EMBE(\hat{X}) = O(u)$ is obvious from the definition.\\
The advantage of EMBE as a measure for the backward error is that it results in a notion of stability, i.e.\ \emph{element-wise mixed backward stability}, which is stronger than norm-wise backward stability and can provably be obtained for quadratic polynomials \cite{mastronardi2015revisiting}. A drawback of this measure is that it is \emph{hard to compute} $\EMBE(\hat{X})$ for a given set of approximate solutions $\hat{X}$ because of the rather abstract definition. In \cite{vanbarel2019tropical}, Tisseur and Van Barel define the \emph{min-max element-wise backward error} of $\hat{X}$ as 
$$ \TBE(\hat{X}) = \max_{i=0, \ldots, d-1} \left | \frac{c_i-\hat{c}_i}{r_i c_i} \right |$$
where $r_i \geq 1$ are constants that can be computed \emph{in linear time} from the coefficients of $f$. The $r_i$ depend only on the tropical roots of $f$, which is why we will refer to this error measure as the \emph{tropical backward error} (TBE). We will give a definition of the numbers $r_i$ in Section \ref{sec:generald}. The authors of \cite{vanbarel2019tropical} also provide an algorithm that, under some assumptions on the numerical behavior of a modified QZ-algorithm (see \cite[Section 5, Assumption 1]{vanbarel2019tropical}), computes a set of approximate roots $\hat{X}$ satisfying $\TBE(\hat{X}) = O(u)$.

In this paper, we investigate the relations between the TBE and the EMBE. In particular, we show that these error measures are equivalent under suitable assumptions. Here's a simplified version of our first main theorem. 
\begin{theorem} \label{thm:simplethm1}
Assume that the tropical roots of $f$ are of the same order of magnitude as the corresponding classical roots and $|\hat{x}_j|$ are of the same order of magnitude as $|x_j|$. Then we have that $\EMBE(\hat{X}) = O(u)$ implies $\TBE(\hat{X}) = O(u)$.
\end{theorem}
The strategy for proving Theorem \ref{thm:simplethm1} will also allow us to prove that, under the assumptions of the theorem, $\EMBE(\hat{X}) = O(u)$ implies $\NBE(\hat{X}) = O(u)$. This was proved in \cite{mastronardi2015revisiting} for the case where $d = 2$. Under some stronger assumptions, we also prove the reverse implication.
\begin{theorem} \label{thm:simplethm2}
Assume that the tropical roots of $f$ are of the same order of magnitude as the corresponding classical roots and $|\hat{x}_j|$ are of the same order of magnitude as $|x_j|$. Moreover, assume that for each $x_j \in X$, there are two terms $c_{\beta'}x_j^{\beta'}$ and $c_{\beta}x_j^\beta$ of $f(x_j)$ such that 
$$ |c_ix_j^i| \ll |c_{\beta'}x_j^{\beta'}| \quad \text{and} \quad |c_ix_j^i| \ll |c_{\beta}x_j^\beta|, \quad \text{for all } i \neq \beta', \beta.$$
Then we have that $\TBE(\hat{X}) = O(u)$ implies $\EMBE(\hat{X}) = O(u)$.
\end{theorem}
We will give numerical evidence that Theorem \ref{thm:simplethm2} holds without the extra assumption on the polynomial $f$. In summary, we have the following diagram, where the arrows are implications.
\begin{equation} \label{eq:diagram}
 \begin{tikzcd}
 \EBE(\hat{X}) = O(u) \arrow[r] \arrow[d] \arrow[dr] & \TBE(\hat{X}) = O(u) \arrow[dl,mycolor1] \arrow[d,mycolor1,dashed] \\
 \NBE(\hat{X}) = O(u) & \EMBE(\hat{X}) = O(u) \arrow[u,mycolor1,shift left = 0.2cm] \arrow[l,mycolor1]
 \end{tikzcd}
\end{equation}
Here, the black arrows are implications which are obvious from the definitions. The blue arrows are implications that we prove under the assumptions of Theorem \ref{thm:simplethm1}. The dashed arrow represents Theorem \ref{thm:simplethm2}, which uses stronger assumptions. 

This article is organized as follows. In the next section we prove the equivalence of the tropical and element-wise mixed backward error measures in the case where $d = 2$. This can be seen as an extension of the analysis in \cite{mastronardi2015revisiting}, and it is instructive for the general case. The proofs for general $d$ are given in Section \ref{sec:generald}. In Section \ref{sec:numexp} we show some computational experiments and give numerical evidence that Theorem \ref{thm:simplethm2} holds under weaker assumptions.

\section{Backward Error for Quadratic Polynomials} \label{sec:d2}

In this section, we prove that the \emph{element-wise mixed backward error} (EMBE) as introduced by Mastronardi and Van Dooren \cite{mastronardi2015revisiting} and the \emph{tropical backward error} (TBE) from \cite{vanbarel2019tropical} are equivalent backward error measures for the roots of a quadratic polynomial 
\begin{equation*}\label{eq:quadratic-poly}
 f = ax^2 + bx + c = a(x-x_1)(x-x_2).
\end{equation*}
For simplicity, we assume that $a,b,c \in \C^* = \C \setminus \{0\}$. 
For the approximate roots $ \hat{X} = \{\hat{x}_1, \hat{x}_2 \}$ of $f$, $\EMBE(\hat{X})$ is the smallest number $\varepsilon \geq 0$ such that there exists $\tilde{X}= \{\tilde{x}_1, \tilde{x}_2 \}$ with 
\begin{align*}  
\tilde{f} = a(x-\tilde{x}_1)(x-\tilde{x}_2) &= ax^2 + \tilde{b} x + \tilde{c},\\
    |\hat{x}_1- \tilde{x}_1| \leq \varepsilon |\tilde{x}_1|, \quad &|\hat{x}_2- \tilde{x}_2| \leq \varepsilon |\tilde{x}_2|, \\
    |b - \tilde{b}| \leq \varepsilon |b|, \quad &|c - \tilde{c}| \leq \varepsilon |c|.
\end{align*}
In analogy with \cite{vanbarel2019tropical}, we define the $\TBE(\hat{X})$ to be the smallest number $\varepsilon \geq 0$ such that 
\begin{align*}
\hat{f} = a(x-\hat{x}_1)(x-\hat{x}_2) = ax^2 + \hat{b} x + \hat{c}\\
|b - \hat{b}| \leq r_b \varepsilon |b|, \quad |c - \hat{c}| \leq \varepsilon |c|,
\end{align*}
where $r_b = \max(1, \sqrt{|ac|}/|b|)$. The definition of $r_b$ will be clarified in Section \ref{sec:generald}. For the approximate roots $\hat{x}_j$ and $\tilde{x}_j$ in these definitions, we will assume that the order of magnitude of $|\tilde{x}_j|$ and $|\hat{x}_j|$ is the same as the order of magnitude of $|x_j|$ (that is, we allow relative errors of size at most 1). Note that by the definition of EMBE, it is sufficient that this is satisfied for $|\hat{x}_j|$. 

We will now relate these two error measures. Let $\varepsilon = \EMBE(\hat{X})$. We observe 
\begin{align*}
\frac{\hat{c} - c}{c} &= \frac{\hat{c} - \tilde{c}}{c} + \frac{\tilde{c} - c}{c} = \frac{\hat{x}_1 \hat{x}_2 - \tilde{x}_1 \tilde{x}_2}{x_1x_2} + \frac{\tilde{c} - c}{c} \\
&= \frac{(\tilde{x}_1 + (\hat{x}_1 - \tilde{x}_1))(\tilde{x}_2 + (\hat{x}_2 - \tilde{x}_2)) - \tilde{x}_1 \tilde{x}_2}{x_1x_2}   + \frac{\tilde{c} - c}{c} \,.
\end{align*}
Since $|c - \tilde{c}| \leq \varepsilon |c|$, we have $|1-\frac{\tilde{x}_1 \tilde{x}_2}{x_1x_2}| \leq \varepsilon$ and it follows
$$\left | \frac{\hat{c} - c}{c} \right | \leq (2 \varepsilon + \varepsilon^2) \left | \frac{\tilde{x}_1\tilde{x}_2}{x_1x_2} \right | + \varepsilon \lesssim 3\varepsilon \,.$$
For the coefficient $\hat{b}$, we find in an analogous way that 
\begin{equation}\label{eq:quad-b}
 \left | \frac{\hat{b} - b}{b}  \right | \leq \varepsilon \left ( 1 + \frac{|\tilde{x}_1| + |\tilde{x}_2|}{|x_1 + x_2|} \right ) .
\end{equation}  
If the solutions $x_1$ and $x_2$ have different orders of magnitude, there does not occur any cancellation in the denominator of the right hand side of \eqref{eq:quad-b} and this implies $\left | \frac{\hat{b} - b}{b}  \right | \lesssim 2 \varepsilon$. However, if the order of magnitude of both solutions is the same, the factor standing with $\varepsilon$ may be significantly larger than 1 due to cancellation. We will now make this precise and relate this to the number $r_b$. 
We define $\gamma = x_1/x_2$. By the assumption that $|\tilde{x}_i|$ is of the same order of magnitude as $|x_i|$, \eqref{eq:quad-b} can be written as
\begin{equation} \label{eq:quad-ineq0}
\left | \frac{\hat{b} - b}{b}  \right | \leq \varepsilon \left ( 1 + K \frac{|\gamma| + 1}{|\gamma + 1|} \right )
\end{equation}
with $K$ a small constant. We assume, without loss of generality, that $0 < |\gamma| \leq 1$. Note that we have for $0<|\gamma|<1$ the inequality
\begin{equation} \label{eq:quad-ineq1}
\frac{|\gamma| + 1}{|\gamma + 1|} \leq \frac{|\gamma| + 1}{||\gamma| - 1|} = \frac{|\gamma| + 1}{1-|\gamma|}, 
\end{equation}
which follows from $|x-y| \geq ||x| - |y||, \forall x, y \in \C$ applied to $x = \gamma, y = -1$. Assume that 
$$\frac{\sqrt{|ac|}}{|b|} = \frac{\sqrt{|\gamma|}}{|\gamma+1|} \leq 1.$$
In this case, we have 
\begin{equation} \label{eq:quad-ineq2}
\frac{|\gamma| + 1}{|\gamma + 1|} \leq \frac{|\gamma| + 1}{\sqrt{|\gamma|}}.
\end{equation}
Now, assume 
$$\frac{\sqrt{|ac|}}{|b|} = \frac{\sqrt{|\gamma|}}{|\gamma+1|} \geq 1.$$
In this case
\begin{equation} \label{eq:quad-ineq3}
 \frac{|\gamma| + 1}{|\gamma + 1|} \left ( \frac{\sqrt{|ac|}}{|b|} \right )^{-1}  \leq \frac{|\gamma| + 1}{|\gamma + 1|} \leq  \frac{|\gamma| + 1}{1-|\gamma|}.
 \end{equation}
Also, 
\begin{equation} \label{eq:quad-ineq4}
\frac{|\gamma| + 1}{|\gamma + 1|} \left ( \frac{\sqrt{|ac|}}{|b|} \right )^{-1} = \frac{|\gamma| + 1}{|\gamma + 1|} \left ( \frac{\sqrt{|\gamma|}}{|\gamma+1|}  \right )^{-1} = \frac{|\gamma| + 1}{\sqrt{|\gamma|}}.
\end{equation} 
Using \eqref{eq:quad-ineq1}-\eqref{eq:quad-ineq4}, we find that 
$$\frac{|\gamma| + 1}{|\gamma + 1|} r_b^{-1} \leq \min \left ( \frac{|\gamma| + 1}{1-|\gamma|}, \frac{|\gamma| + 1}{\sqrt{|\gamma|}} \right ),$$
which gives
$$\frac{|\gamma| + 1}{|\gamma + 1|} r_b^{-1} \leq \begin{cases} \frac{\alpha + 1}{1 - \alpha} = \sqrt{5} & 0 < |\gamma| \leq \alpha\\
\frac{\alpha + 1}{\sqrt{\alpha}} = \sqrt{5} & \alpha \leq |\gamma| \leq 1
\end{cases} \quad \Rightarrow ~ \frac{|\gamma| + 1}{|\gamma + 1|}r_b^{-1} \leq \sqrt{5},$$
where $\alpha = \frac{3}{2} - \frac{\sqrt{5}}{2}$. This is illustrated in Figure \ref{fig:quadbounds}.
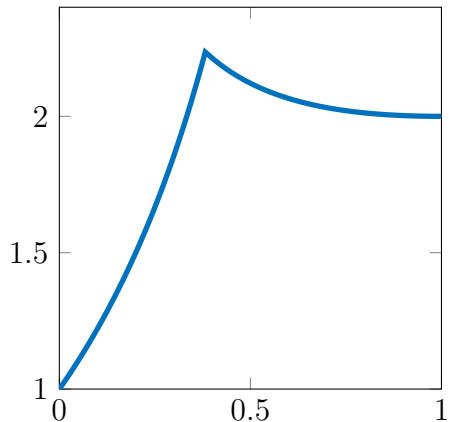
\begin{figure}[h!]
\centering
%
%
\definecolor{mycolor1}{rgb}{0.00000,0.44700,0.74100}%
\begin{tikzpicture}

\begin{axis}[%
width=2in,
height=2in,
at={(0.772in,0.516in)},
scale only axis,
xmin=0,
xmax=1,
xtick = {0,0.5,1},
ymin=1,
ymax=2.4,
axis background/.style={fill=white}
]
\addplot [color=mycolor1, line width=2.0pt, forget plot]
  table[row sep=crcr]{%
0	1\\
0.0050251256281407	1.01010101010101\\
0.0100502512562814	1.02030456852792\\
0.0150753768844221	1.03061224489796\\
0.0201005025125628	1.04102564102564\\
0.0251256281407035	1.05154639175258\\
0.0301507537688442	1.06217616580311\\
0.0351758793969849	1.07291666666667\\
0.0402010050251256	1.08376963350785\\
0.0452261306532663	1.09473684210526\\
0.050251256281407	1.10582010582011\\
0.0552763819095477	1.11702127659574\\
0.0603015075376884	1.1283422459893\\
0.0653266331658292	1.13978494623656\\
0.0703517587939698	1.15135135135135\\
0.0753768844221105	1.16304347826087\\
0.0804020100502513	1.17486338797814\\
0.085427135678392	1.18681318681319\\
0.0904522613065327	1.19889502762431\\
0.0954773869346734	1.21111111111111\\
0.100502512562814	1.22346368715084\\
0.105527638190955	1.23595505617978\\
0.110552763819095	1.24858757062147\\
0.115577889447236	1.26136363636364\\
0.120603015075377	1.27428571428571\\
0.125628140703518	1.28735632183908\\
0.130653266331658	1.30057803468208\\
0.135678391959799	1.31395348837209\\
0.14070351758794	1.32748538011696\\
0.14572864321608	1.34117647058824\\
0.150753768844221	1.35502958579882\\
0.155778894472362	1.36904761904762\\
0.160804020100503	1.38323353293413\\
0.165829145728643	1.39759036144578\\
0.170854271356784	1.41212121212121\\
0.175879396984925	1.42682926829268\\
0.180904522613065	1.44171779141104\\
0.185929648241206	1.45679012345679\\
0.190954773869347	1.47204968944099\\
0.195979899497487	1.4875\\
0.201005025125628	1.50314465408805\\
0.206030150753769	1.51898734177215\\
0.21105527638191	1.53503184713376\\
0.21608040201005	1.55128205128205\\
0.221105527638191	1.56774193548387\\
0.226130653266332	1.58441558441558\\
0.231155778894472	1.60130718954248\\
0.236180904522613	1.61842105263158\\
0.241206030150754	1.63576158940397\\
0.246231155778894	1.65333333333333\\
0.251256281407035	1.67114093959732\\
0.256281407035176	1.68918918918919\\
0.261306532663317	1.70748299319728\\
0.266331658291457	1.72602739726027\\
0.271356783919598	1.7448275862069\\
0.276381909547739	1.76388888888889\\
0.281407035175879	1.78321678321678\\
0.28643216080402	1.80281690140845\\
0.291457286432161	1.82269503546099\\
0.296482412060302	1.84285714285714\\
0.301507537688442	1.86330935251799\\
0.306532663316583	1.88405797101449\\
0.311557788944724	1.90510948905109\\
0.316582914572864	1.92647058823529\\
0.321608040201005	1.94814814814815\\
0.326633165829146	1.97014925373134\\
0.331658291457286	1.99248120300752\\
0.336683417085427	2.01515151515152\\
0.341708542713568	2.0381679389313\\
0.346733668341709	2.06153846153846\\
0.351758793969849	2.08527131782946\\
0.35678391959799	2.109375\\
0.361809045226131	2.13385826771654\\
0.366834170854271	2.15873015873016\\
0.371859296482412	2.184\\
0.376884422110553	2.20967741935484\\
0.381909547738693	2.23577235772358\\
0.386934673366834	2.22965249709062\\
0.391959798994975	2.22334024617849\\
0.396984924623116	2.21719919411749\\
0.402010050251256	2.21122365450699\\
0.407035175879397	2.20540819335927\\
0.412060301507538	2.19974761502319\\
0.417085427135678	2.19423694905119\\
0.422110552763819	2.18887143793587\\
0.42713567839196	2.18364652564929\\
0.4321608040201	2.17855784692355\\
0.437185929648241	2.17360121721702\\
0.442211055276382	2.16877262331486\\
0.447236180904523	2.16406821451735\\
0.452261306532663	2.15948429437303\\
0.457286432160804	2.15501731291746\\
0.462311557788945	2.15066385938152\\
0.467336683417085	2.14642065533614\\
0.472361809045226	2.14228454824306\\
0.477386934673367	2.13825250538346\\
0.482412060301508	2.1343216081388\\
0.487437185929648	2.13048904659992\\
0.492462311557789	2.12675211448249\\
0.49748743718593	2.12310820432849\\
0.50251256281407	2.11955480297492\\
0.507537688442211	2.1160894872725\\
0.512562814070352	2.11270992003806\\
0.517587939698492	2.109413846226\\
0.522613065326633	2.10619908930475\\
0.527638190954774	2.10306354782562\\
0.532663316582915	2.1000051921719\\
0.537688442211055	2.09702206147747\\
0.542713567839196	2.09411226070423\\
0.547738693467337	2.09127395786918\\
0.552763819095477	2.08850538141188\\
0.557788944723618	2.08580481769422\\
0.562814070351759	2.08317060862462\\
0.5678391959799	2.08060114939951\\
0.57286432160804	2.07809488635523\\
0.577889447236181	2.07565031492419\\
0.582914572864322	2.07326597768922\\
0.587939698492462	2.07094046253077\\
0.592964824120603	2.06867240086165\\
0.597989949748744	2.06646046594455\\
0.603015075376884	2.06430337128781\\
0.608040201005025	2.06219986911516\\
0.613065326633166	2.06014874890547\\
0.618090452261307	2.05814883599877\\
0.623115577889447	2.05619899026501\\
0.628140703517588	2.0542981048323\\
0.633165829145729	2.05244510487146\\
0.638190954773869	2.05063894643405\\
0.64321608040201	2.04887861534109\\
0.648241206030151	2.04716312611987\\
0.653266331658292	2.04549152098646\\
0.658291457286432	2.04386286887161\\
0.663316582914573	2.04227626448782\\
0.668341708542714	2.04073082743572\\
0.673366834170854	2.03922570134758\\
0.678391959798995	2.03776005306637\\
0.683417085427136	2.03633307185853\\
0.688442211055276	2.0349439686588\\
0.693467336683417	2.03359197534576\\
0.698492462311558	2.03227634404639\\
0.703517587939699	2.03099634646847\\
0.708542713567839	2.02975127325944\\
0.71356783919598	2.02854043339052\\
0.718592964824121	2.02736315356487\\
0.723618090452261	2.02621877764883\\
0.728643216080402	2.02510666612498\\
0.733668341708543	2.02402619556632\\
0.738693467336683	2.02297675813029\\
0.743718592964824	2.02195776107203\\
0.748743718592965	2.02096862627593\\
0.753768844221106	2.02000878980458\\
0.758793969849246	2.01907770146452\\
0.763819095477387	2.01817482438791\\
0.768844221105528	2.01729963462961\\
0.773869346733668	2.0164516207788\\
0.778894472361809	2.0156302835847\\
0.78391959798995	2.01483513559571\\
0.78894472361809	2.01406570081141\\
0.793969849246231	2.01332151434695\\
0.798994974874372	2.01260212210919\\
0.804020100502513	2.01190708048421\\
0.809045226130653	2.01123595603578\\
0.814070351758794	2.01058832521409\\
0.819095477386935	2.00996377407466\\
0.824120603015075	2.00936189800678\\
0.829145728643216	2.00878230147117\\
0.834170854271357	2.00822459774654\\
0.839195979899497	2.00768840868462\\
0.844221105527638	2.00717336447339\\
0.849246231155779	2.00667910340821\\
0.85427135678392	2.00620527167038\\
0.85929648241206	2.0057515231131\\
0.864321608040201	2.00531751905434\\
0.869346733668342	2.00490292807641\\
0.874371859296482	2.0045074258321\\
0.879396984924623	2.00413069485687\\
0.884422110552764	2.00377242438721\\
0.889447236180904	2.00343231018461\\
0.894472361809045	2.00311005436519\\
0.899497487437186	2.00280536523463\\
0.904522613065327	2.0025179571282\\
0.909547738693467	2.00224755025587\\
0.914572864321608	2.00199387055211\\
0.919597989949749	2.00175664953031\\
0.924623115577889	2.00153562414171\\
0.92964824120603	2.00133053663849\\
0.934673366834171	2.00114113444116\\
0.939698492462312	2.00096717000978\\
0.944723618090452	2.00080840071915\\
0.949748743718593	2.00066458873761\\
0.954773869346734	2.00053550090953\\
0.959798994974874	2.00042090864118\\
0.964824120603015	2.00032058778995\\
0.969849246231156	2.00023431855688\\
0.974874371859296	2.00016188538218\\
0.979899497487437	2.00010307684386\\
0.984924623115578	2.00005768555923\\
0.989949748743719	2.00002550808925\\
0.994974874371859	2.00000634484553\\
1	2\\
};
\end{axis}
\end{tikzpicture}%
\caption{The value of $\min \left ( \frac{|\gamma| + 1}{1-|\gamma|}, \frac{|\gamma| + 1}{\sqrt{|\gamma|}} \right )$ for $0 \leq |\gamma| \leq 1$.}
\label{fig:quadbounds}
\end{figure}
It follows immediately from this observation and \eqref{eq:quad-ineq0} that 
$$  \left | \frac{\hat{b} - b}{b}  \right | \leq \varepsilon \left (\sqrt{5}K + r_b^{-1} \right ) r_b \leq \varepsilon (\sqrt{5}K+1) r_b.$$
Figure \ref{fig:levelplots} illustrates the values of $r_b$ and $\frac{|\gamma| + 1}{|\gamma + 1|} r_b^{-1}$ as a function of $\gamma$ in the unit disk. 
\begin{figure}[h!]
\centering
\includegraphics[scale=0.5]{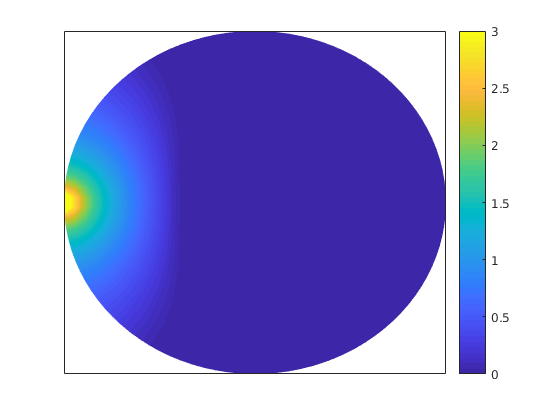}
\includegraphics[scale=0.5]{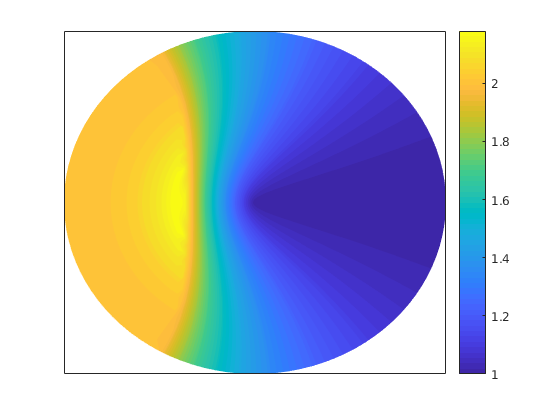}
\caption{Left: illustration of the value of $\log(r_b)$ as a function of $\gamma$ for $0 \leq |\gamma| \leq 1$. Right: illustration of the value of $\frac{|\gamma| + 1}{|\gamma + 1|} r_b^{-1}$ for $0 \leq |\gamma| \leq 1$.}
\label{fig:levelplots}
\end{figure}
This shows that element-wise mixed backward stability implies tropical backward stability. The converse also holds, as we will now show.
Let $\TBE(\hat{X}) = \varepsilon$ for the computed roots $\hat{X} = \{\hat{x}_1, \hat{x}_2\}$.
If $r_b = 1$ then we can take $\tilde{x}_1 = \hat{x}_1$ and $\tilde{x}_2 = \hat{x}_2$ such that $\EMBE(\hat{X}) \leq \TBE(\hat{X})$.
Suppose now that $r_b = \sqrt{|ac|}/|b| > 1$ and without loss of generality assume $|\hat{x}_1| \le |\hat{x}_2|$.
$\TBE(\hat{X}) = \varepsilon$ implies that there exists $\hat{\delta}_b \in \C$ with $|\hat{\delta}_b| \leq \varepsilon$ such that
$$- \hat{b} = a(\hat{x}_1 + \hat{x}_2) =  -b (1 + r_b \hat{\delta}_b) \,.$$
Let $\tilde{x}_1 = \hat{x}_1 + \frac{b}{a} r_b \hat{\delta}_b$ and $\tilde{x}_2 = \hat{x}_2$.
Then we have
\begin{equation*}
  - \tilde{b} = a(\tilde{x}_1 + \tilde{x}_2) =  a(\hat{x}_1 + \frac{b}{a} r_b \hat{\delta}_b + \hat{x}_2)
  = a(\hat{x}_1 + \hat{x}_2)  + b r_b \hat{\delta}_b 
  = -b (1 + r_b \hat{\delta}_b) + b r_b \hat{\delta}_b = -b \,.
\end{equation*}
Note that
$$
|\tilde{x}_1 - \hat{x}_1| = \left |\frac{b}{a} r_b \hat{\delta}_b \right | = \sqrt{\left| \frac{c}{a}\right|}  |\hat{\delta}_b| = \sqrt{|x_1 x_2|} |\hat{\delta}_b| \lesssim |\tilde{x}_1| \varepsilon \,.
$$
We also have
$$
\tilde{c} -\hat{c} = a \tilde{x}_1 \tilde{x}_2  - a\hat{x}_1 \hat{x}_2 = a (\hat{x}_1 + \frac{b}{a} r_b \hat{\delta}_b) \hat{x}_2  - a\hat{x}_1 \hat{x}_2 = b r_b \hat{\delta}_b \hat{x}_2
$$
from which we get
$$|\tilde{c} -\hat{c}| = |a| \left| \frac{b}{a} r_b \hat{\delta}_b \hat{x}_2 \right | \lesssim \varepsilon |a||\tilde{x}_1 \tilde{x}_2| = \varepsilon | \tilde{c}| \,.$$
We conclude that tropical backward stability implies element-wise mixed backward stability. 
\begin{remark}
We assumed $a,b,c \in \C^*$ in this discussion. If $a = 0$, we are solving a linear equation and there is nothing to prove. If $c = 0$, the root $x_1 = 0$ can be deflated and we are again left with a linear equation. If $b = 0$, a similar derivation can be made. We omit the details but give a brief outline. We replace the conditions on $\tilde{b}$ and $\hat{b}$ in the definitions of $\EMBE(\hat{X})$ and $\TBE(\hat{X})$ respectively by $|\tilde{b}| \leq \varepsilon$ and $|\hat{b}| \leq r_b \varepsilon$ where $r_b = \max(1, \sqrt{|ac|}) = \max(1,|a x_1|)$ (note that in this case $|x_1| = |x_2|$). One derives bounds in a similar way for the implication $\EMBE(\hat{X}) = \varepsilon \Rightarrow \TBE(\hat{X}) = O(\varepsilon)$. For the other implication, there is again nothing to prove when $r_b = 1$. When $r_b = |ax_1|$ we observe that we can write $\hat{b} = - r_b \hat{\delta}_b$ with $|\hat{\delta}_b| \leq \varepsilon$ and we set $\tilde{x}_1 = \hat{x}_1 + a^{-1}r_b \hat{\delta}_b, \tilde{x}_2 = \hat{x}_2$. 
\end{remark}
\noindent We arrive at the following statement. 
\begin{proposition}
Let $\hat{X} = \{\hat{x}_1, \hat{x}_2 \}$ be a set of approximations for the roots $X = \{x_1,x_2 \}$ of a quadratic polynomial $f = ax^2 + bx + c \in \C[x]$, where $a,c \neq 0$. Under the assumption that $|\hat{x}_i|$ has the same order of magnitude as $|x_i|$, $i=1,2$, we have that 
$$ \EMBE(\hat{X}) = O(u) \quad \textup{if and only if} \quad \TBE(\hat{X}) = O(u).$$
\end{proposition}

\section{Backward Error for Polynomials of General Degree} \label{sec:generald}
We now generalize the results of the previous section to polynomials of arbitrary degree $d$.
In the following let $f = \sum_{i=0}^d c_i x^i \in \C[x]$, $c_0, c_d \ne 0$, be a polynomial of degree $d$ with roots $X = \{x_1,\ldots,x_d \} \subset \C^*$ and let $\hat{X} = \{ \hat{x}_1, \ldots, \hat{x}_d \} \subset \C^*$ be the approximate roots. Without loss of generality we assume that the roots are labeled such that $|x_1| \le |x_2| \le \ldots \le |x_d|$.

We first formally generalize the notion of the element-wise mixed backward error due to Mastronardi and Van Dooren \cite{mastronardi2015revisiting} from quadratic polynomials to general polynomials of degree $d$.

\begin{definition}[Element-wise mixed backward error] \label{def:EMBE}
  The \emph{element-wise mixed backward error} of $\hat{X}$, denoted $\EMBE(\hat{X})$, is the smallest number $\varepsilon \geq 0$ such that there exist points $\tilde{X} = \{\tilde{x}_1,\ldots, \tilde{x}_d \}\subset \C^*$ satisfying $|\hat{x}_j - \tilde{x}_j| \le \varepsilon |\tilde{x}_j|$, $j = 1,\ldots,d$, and 
\[ \begin{cases}
    |c_i - \tilde{c}_i| \le \varepsilon |c_i|, & c_i \neq 0, \\
     |c_i - \tilde{c}_i| \le \varepsilon, & c_i = 0,
   \end{cases}
\]
  for all $i=0,\ldots,d-1$, where the $\tilde{c}_i$ are the coefficients of 
  $$\tilde{f} = c_d\prod_{j=1}^d (x - \tilde{x}_j) = c_d x^d + \sum_{i=0}^{d-1} \tilde{c}_{i}x^{i} \,.$$
\end{definition}
Before we can state the generalization of the tropical backward error for degree $d$ polynomials we need to introduce some definitions and concepts.
We define the tropical polynomial $\trop{f}(\tau)$ associated to  $f(x)$ as
\begin{equation} \label{eq:troppol}
\trop{f}: \R \cup \{ -\infty \} \rightarrow  \R \cup \{ -\infty \}, \quad \tau \mapsto  \max_{0 \le i \le d} v_i + i \tau 
\end{equation}
where $v_i = \log|c_i|$. The number $v_i$ is the \emph{valuation} of $c_i$ under the valuation map $\log|\cdot|$. Any base for the logarithm can be used in theory. We want to think of the image under the valuation map as the `order of magnitude' of the modulus of a complex number. In this paper, when we state that $\log|c| \approx 0, c \in \C,$ we mean that $|c|$ is of \emph{order 1}. In tropical geometry the map $\log|\cdot|$ is referred to as an Archimedean valuation. For a general introduction to tropical geometry we refer to \cite{sharify2011scaling,maclagan2015introduction}.
\begin{figure}[h]
  \begin{center}
    \includegraphics[width=0.8\textwidth]{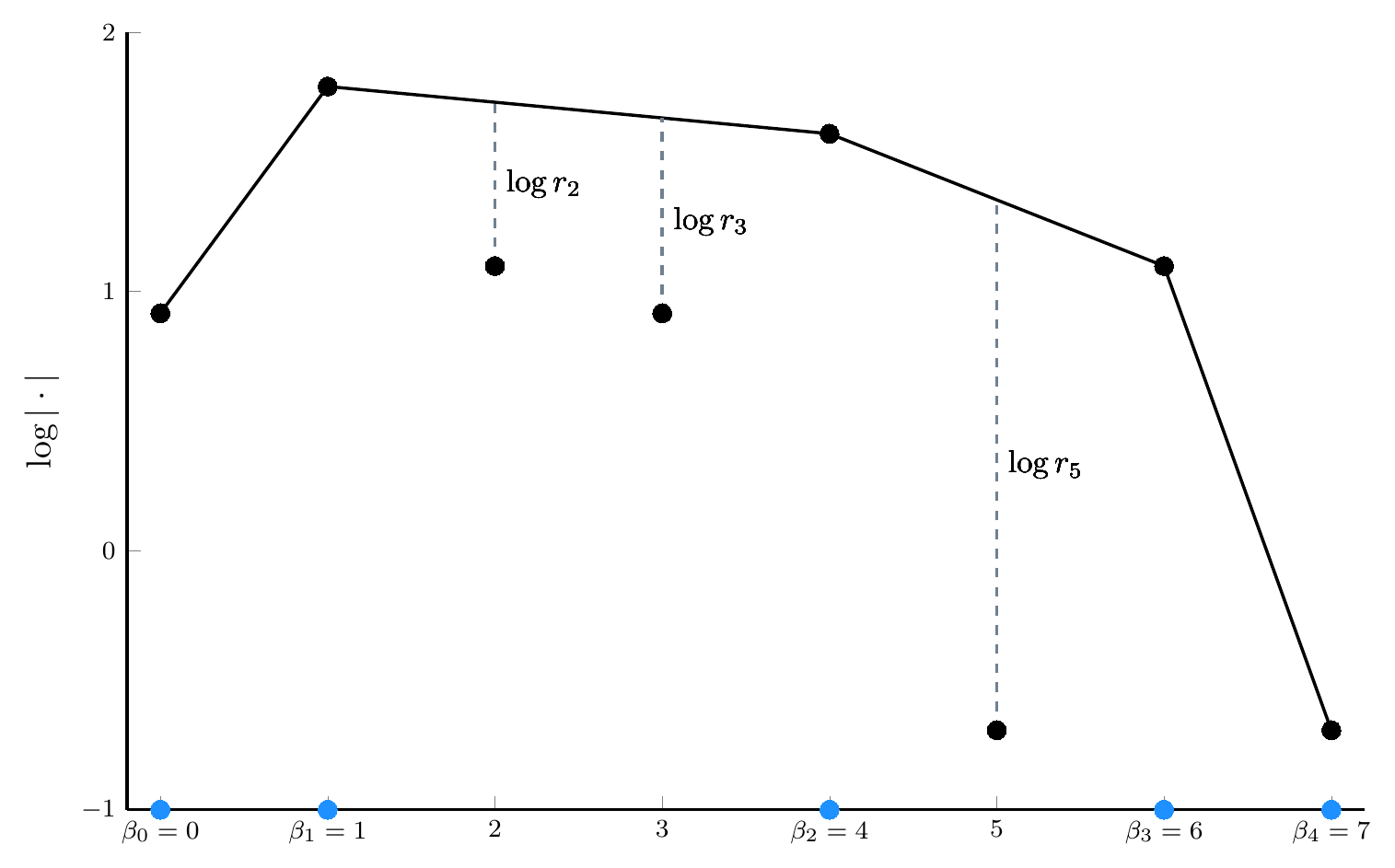}
  \end{center} 
  \caption{Consider $f(x) = \frac12x^7 + 3x^6 + \frac12 x^5 + 5x^4 +  \frac52 x^3 + 3x^2 + 6x + \frac52$.
  The figure depicts the upper convex hull of the lifted Newton polytope, the geometric derivation of the $r_i$ values and the 
  induced subdivision $\Delta = ~\{(\beta_0=0,\beta_1=1), (1, \beta_2=4), (4, \beta_3=6), (6, \beta_4=7)\}$.}\label{fig:newton-subdiv}
\end{figure}

The \emph{Newton polytope} of $f$ is the line segment $[0,d] \subset \R$. The convex hull of the points $\{(i, v_i)\}_{0\le i \le d} \subset \R^2$ is called the \emph{lifted} Newton polytope. We will consider the \emph{upper hull} of the lifted Newton polytope. For a specific example, this is shown as a solid black line in Figure \ref{fig:newton-subdiv}. The vertices of this upper hull are the points $(\beta_\ell, v_{\beta_\ell})$, $\ell=0,\ldots,s$, with
$$0 = \beta_0 < \beta_1 < \ldots < \beta_s = d \,.$$
We call the set $\Delta = \{(\beta_0, \beta_1), (\beta_1, \beta_2), \ldots, (\beta_{s-1}, \beta_s) \}$ the subdivision induced by the coefficients $c_i$, or the \emph{induced subdivision} for short.
We say that a point $\tau \in \R \cup \{ -\infty \}$ is a \emph{root} of $\trop{f}$ if the maximum in \eqref{eq:troppol} is attained at least twice. A root of $\trop{f}$ is called a \emph{tropical root} of $f$. The \emph{multiplicity} of a root $\tau$ of $\trop{f}$ is the number $\beta_\ell - \beta_{\ell-1}$ where $\beta_\ell$ and $\beta_{\ell-1}$ are the largest, respectively the smallest value of $i$ for which $v_i + i \tau$ is maximal.
Counted with multiplicity, $\trop{f}$ has $d$ roots $\tau_1, \ldots, \tau_d$ and we can give a closed formula for them.
For $\beta_{\ell - 1} < i \le \beta_{\ell}$ we have
$$\tau_i = \frac{1}{m_{\ell}}(v_{\beta_{\ell - 1}} - v_{\beta_{\ell}}) =
  \log \; \left|\frac{c_{\beta_{\ell - 1}}}{c_{\beta_{\ell}}}\right|^{\frac{1}{m_\ell}}  $$
where $m_\ell = \beta_{\ell} - \beta_{\ell - 1}$ is the multiplicity.
In particular, the definition implies $$\tau_1 = \tau_{\beta_0} = \cdots = \tau_{\beta_1} < \tau_{\beta_1+1} = \cdots = \tau_{\beta_2} < \cdots <\tau_{\beta_{s-1}+1} = \cdots = \tau_{\beta_s}= \tau_d.$$
Tropical roots of polynomials are used for scaling (matrix) polynomial eigenvalue problems, see for instance \cite{bini2013locating,gaubert2009tropical,noferini2015tropical}.

Furthermore, for $i=0,\ldots,d$ we define the constants
$$r_i = \begin{cases}
  1, & i = \beta_\ell \text{ for some } \ell \\
  \exp(v_{\beta_{\ell}} + (\beta_{\ell}-i) \tau_i - v_i), 
  & \beta_{\ell - 1} < i < \beta_{\ell}  \text{ for some } \ell \text{ and } v_i \in \R \\
  \exp(v_{\beta_{\ell}} + (\beta_{\ell}-i) \tau_i), 
  & \beta_{\ell - 1} < i < \beta_{\ell}  \text{ for some } \ell \text{ and } v_i = -\infty
\end{cases}\,.$$
Geometrically, if $c_i \neq 0$ then $\log r_i$ is the distance of $v_i = \log|c_i|$ to the upper convex hull of the lifted Newton polytope. Figure \ref{fig:newton-subdiv} illustrates these concepts. Note that $\tau_i$, $\beta_{\ell - 1} < i \le \beta_{\ell}$, is the negative slope of the line connecting $(\beta_{\ell-1}, v_{\beta_{\ell-1}})$ and $(\beta_{\ell}, v_{\beta_{\ell}})$. We note that the complexity of computing the tropical roots of $f$ is \emph{linear} in the degree $d$, see e.g.\ \cite[Proposition 2.2.1]{sharify2011scaling}. 

\begin{definition}[Tropical backward error]
  The \emph{tropical backward error} of $\hat{X}$, denoted $\TBE(\hat{X})$, is the smallest number $\varepsilon \ge 0$ such that for all $i=0,\ldots,d-1$
\[ \begin{cases}
    |c_i - \hat{c}_i| \le r_i \varepsilon |c_i|, & c_i \neq 0, \\
     |c_i - \hat{c}_i| \le r_i \varepsilon, & c_i = 0,
   \end{cases}
\]
  where the $\hat{c}_i$ are the coefficients of 
  $$\hat{f} = c_d\prod_{j=1}^d (x - \hat{x}_j) = c_d x^d + \sum_{i=0}^{d-1} \hat{c}_{i}x^{i} \,.$$
\end{definition}

In the following we show that an element-wise mixed backward error of order machine precision also implies 
a tropical backward error of order machine precision and that the converse holds, under suitable assumptions, as well.
As in Section \ref{sec:d2} for the quadratic case, we assume in the following that $|x_j|$, $|\hat{x}_j|$ and $|\tilde{x}_j|$ have the same order of magnitude for $j=1,\ldots,d$.

\subsection{Element-wise Mixed Backward Stability implies Tropical Backward Stability}
We start by showing that an $\EMBE(\hat{X}) = \varepsilon$ implies $\TBE(\hat{X}) = O(\varepsilon)$.
The coefficients of the polynomial $f$ can be considered as functions of the roots $x_1,\ldots,x_d$.
Let
$$\sigma_k(x_1, \ldots, x_d) = \sum_{|I| = k} \prod_{i \in I} x_i$$
be the $k$-th elementary symmetric polynomial. We have the identity
\begin{equation*}
  f = c_d\prod_{i=1}^d (x - x_i) = c_d \sum_{k=0}^d (-1)^{d-k} \sigma_{d-k}(x_1,\ldots,x_d)x^k.
\end{equation*}
\begin{lemma} \label{lem:symmfunc}
If $\EMBE(\hat{X}) = \varepsilon$, then the coefficients of 
$$ \hat{f} = c_d x^d + \sum_{i =0}^{d-1} \hat{c}_i x^i = c_d (x-\hat{x}_1) \cdots (x- \hat{x}_d)$$
satisfy 
\begin{align*}
\left | \frac{\hat{c}_i - c_i}{c_i} \right | &\leq \varepsilon \left ( 1 + (d- i) \frac{\sigma_{d-i}(|\tilde{x}_1|, \ldots, |\tilde{x}_d|)}{|\sigma_{d-i}(x_1,\ldots, x_d)|} + O(\varepsilon) \right ), & \text{when } c_i \neq 0 \\ \intertext{ and }
|\hat{c}_i - c_i| &\leq \varepsilon( 1 + c_d (d-i) \sigma_{d-i}(|\tilde{x}_1|, \ldots, |\tilde{x}_d|) + O(\varepsilon)), & \text{when }c_i = 0.  
\end{align*}
\end{lemma}
\begin{proof}
Suppose $c_i \neq 0$. We have 
$$ \frac{\hat{c}_i - c_i}{c_i} = \frac{\hat{c}_i - \tilde{c}_i}{c_i} + \frac{\tilde{c}_i - c_i}{c_i}.$$
This implies 
\begin{align*}
\left | \frac{\hat{c}_i - c_i}{c_i} \right | &\leq \left |\frac{\hat{c}_i - \tilde{c}_i}{c_i} \right | + \left | \frac{\tilde{c}_i - c_i}{c_i} \right | \\
& \leq \frac{|\sigma_{d-i}(\hat{x}_1, \ldots, \hat{x}_d) - \sigma_{d-i}(\tilde{x}_1, \ldots, \tilde{x}_d)| }{|\sigma_{d-i}(x_1, \ldots, x_d)|} + \varepsilon \\
& = \frac{|\sigma_{d-i}(\tilde{x}_1(1 + \hat{\delta}_1), \ldots, \tilde{x}_d(1+\hat{\delta}_d)) - \sigma_{d-i}(\tilde{x}_1, \ldots, \tilde{x}_d)| }{|\sigma_{d-i}(x_1, \ldots, x_d)|} + \varepsilon
\end{align*}
where $|\hat{\delta}_i| \leq \varepsilon, i = 1, \ldots, d$. Note that the second inequality and the equality both use $\EMBE(\hat{X}) = \varepsilon$. We now observe 
$$\sigma_{d-i}(\tilde{x}_1(1 + \hat{\delta}_1), \ldots, \tilde{x}_d(1+\hat{\delta}_d)) = \sigma_{d-i}(\tilde{x}_1, \ldots, \tilde{x}_d) + \sum_{j=1}^{d-i} \hat{\delta}_j \sigma_{d-i}(\tilde{x}_1, \ldots, \tilde{x}_d) + \text{h.o.t.},$$
where the `higher order terms' contain at least two of the $\hat{\delta}_i$. This, together with the triangle inequality, shows 
$$\left | \frac{\hat{c}_i - c_i}{c_i} \right | \leq \varepsilon (d-i) \frac{\sigma_{d-i}(|\tilde{x}_1|, \ldots, |\tilde{x}_d|)}{|\sigma_{d-i}(x_1, \ldots, x_d)|} + O(\varepsilon^2) + \varepsilon.$$
The case $c_i = 0$ is completely analogous.
\end{proof}

It is well known that the values $\exp(\tau_i)$ are related to the modulus of the classical roots $x_i$, see for instance \cite{sharify2011scaling}. In what follows, we will make the assumption that the order of magnitude of $\exp(\tau_i)$ is equal to that of $|x_i|$ (and, by our previous assumption, also to $|\hat{x}_i|$ and $|\tilde{x}_i|$). Under this assumption, we have for $\beta_{\ell-1}< i \le \beta_{\ell}$ that 
$$ \sigma_{d-i}(|x_1|, \ldots, |x_d|) = \sum_{|I| = d-i} \prod_{j \in I} |x_j| = D_i \Binom{m_{\ell}}{\beta_\ell - i} \prod_{k=0}^{d-i-1} \exp(\tau_{d-k}),$$
with $D_i$ a not too large constant and $$\Binom{m_{\ell}}{\beta_\ell - i} = \frac{m_\ell!}{(\beta_\ell - i)!(i- \beta_{\ell-1})!}$$
the binomial coefficient. This can be seen as follows. The important terms in the expansion of $\sigma_{d-i}(|x_1|, \ldots, |x_d|)$ are those containing $d-i$ large roots. By the ordering of the roots by their modulus, these terms have the order of magnitude of $\prod_{k=0}^{d-i-1} \exp(\tau_{d-k})$. We can assume that each of these terms contains the largest $m_{\ell+1} + \cdots + m_s \leq d-i$ roots. For the remaining factors, we can choose $\beta_\ell - i$ roots among $\{x_{\beta_{\ell-1} +1}, \ldots, x_{\beta_\ell} \}$. 

By our assumption that $|\tilde{x}_i| \approx |x_i|$, we have
$$ \sigma_{d-i}(|\tilde{x}_1|, \ldots, |\tilde{x}_d|) = \sum_{|I| = d-i} \prod_{j \in I} |\tilde{x}_j| = \tilde{D}_i \Binom{m_{\ell}}{\beta_\ell - i} \prod_{k=0}^{d-i-1} \exp(\tau_{d-k}),$$
with $\tilde{D}_i$ a not too large constant.
Taking valuations on both sides of this equality, we get 
\begin{equation} \label{eq:assumption}
\log |\sigma_{d-i}(|\tilde{x}_1|, \ldots, |\tilde{x}_d|)| = w_i + \sum_{k=0}^{d-i - 1} \tau_{d-k}
\end{equation}
where $w_i = \log \left |\tilde{D}_i \binom{m_i}{\beta_\ell - i} \right|$ and $\exp(w_i)$ is a not too large positive number. Equation \eqref{eq:assumption} is the assumption we will use in the next theorem. 

\begin{theorem} \label{thm:EMBEimpliesTBE}
If $\EMBE(\hat{X}) = \varepsilon$ and the order of magnitude of $|x_i|$ is equal to that of $|\hat{x_i}|$ and $\exp(\tau_i)$, $i = 1, \ldots, d$ and \eqref{eq:assumption} holds, then the coefficients of 
$$ \hat{f} = c_d x^d + \sum_{i =0}^{d-1} \hat{c}_i x^i = c_d (x-\hat{x}_1) \cdots (x- \hat{x}_d)$$
satisfy 
\begin{align*}
\left | \frac{\hat{c}_i - c_i}{c_i} \right | &\leq \varepsilon \left ( 1 + (d- i) \exp(w_i) r_i + O(\varepsilon) \right ) \quad \text{when } c_i \neq 0 \\ \intertext{ and }
|\hat{c}_i - c_i| & \leq \varepsilon \left ( 1 + (d- i) \exp(w_i) r_i + O(\varepsilon) \right ) \quad \text{when } c_i = 0.
\end{align*}
In particular, $ \TBE(\hat{X}) \leq \varepsilon \max_i (1 + (d- i) \exp(w_i) + O(\varepsilon))$.
\end{theorem}
\begin{proof}
For $\beta_{\ell-1} \leq i \leq \beta_{\ell}$, assume $c_i \neq 0$. Note that 
\begin{align*}
\log|r_i| &= v_{\beta_{\ell}} + (\beta_{\ell} - i)\tau_{\beta_{\ell}} - v_{i}\\
&= v_{\beta_{\ell+1}} + m_{i+1} \tau_{\beta_{\ell+1}} + (\beta_{\ell} - i)\tau_{\beta_{\ell}} - v_{i} \\
&= v_{\beta_{\ell+2}} + m_{i+2} \tau_{\beta_{\ell+2}} + m_{i+1} \tau_{\beta_{\ell+1}} + (\beta_{\ell} - i)\tau_{\beta_{s}} - v_{i} \\
&= \ldots \\
&= v_{\beta_\ell} + \sum_{k = \ell +1}^\ell m_k \tau_{\beta_k} + (\beta_{\ell} - i)\tau_{\beta_{\ell}} - v_{i} \\
&= v_{d} + \sum_{k=0}^{d-i - 1} \tau_{d-k} - v_i.
\end{align*}
Now, using \eqref{eq:assumption} and $v_i = v_{d} + \log|\sigma_{d-i}(x_1, \ldots, x_d)|$ we get 
$$ \log|r_i| = \log|\sigma_{d-i}(|\tilde{x}_1|, \ldots, |\tilde{x}_d|)| -\log|\sigma_{d-i}(x_1, \ldots, x_d)| - w_i.$$
Therefore 
$$\exp(w_i) r_i = \exp(w_i + \log|r_i|) = \frac{\sigma_{d-i}(|\tilde{x}_1|, \ldots, |\tilde{x}_d|)}{|\sigma_{d-i}(x_1, \ldots, x_d)|}$$ and we are done by Lemma \ref{lem:symmfunc}. The proof for $c_i = 0$ is analogous. 
\end{proof}
In \cite{mastronardi2015revisiting}, Mastronardi and Van Dooren show that, for $d=2$, $\EMBE(\hat{X}) = O(u)$ implies $\NBE(\hat{X}) = O(u)$, where $\NBE(\hat{X})$ is the \emph{norm-wise backward error} as defined in the introduction. Theorem \ref{thm:EMBEimpliesTBE} allows us to prove this statement for general degrees. 
\begin{proposition} \label{prop:EMBEimpliesNBE}
If $\EMBE(\hat{X}) = \varepsilon$ and the assumptions of Theorem \ref{thm:EMBEimpliesTBE} are satisfied, we have that 
$$ \norm{(c_0,\ldots,c_{d-1},c_d) - (\hat{c}_0, \ldots, \hat{c}_{d-1}, c_d)}_2 = O(\varepsilon) \norm{(c_0,\ldots, c_{d-1},c_{d})}_2.$$
\end{proposition}
\begin{proof}
Suppose that $c_i \neq 0$.
We have 
\begin{align*}
|c_i - \hat{c}_i| &\leq |c_i - \tilde{c}_i| + |\hat{c}_i - \tilde{c}_i|.
\end{align*}
Since $\EMBE(\hat{X}) = \varepsilon$, we have that 
$\left | (\hat{c}_i - c_i) - (\hat{c}_i - \tilde{c}_i) \right | \leq \varepsilon |c_i|.$
Combining this with Theorem \ref{thm:EMBEimpliesTBE}, we have that 
$$| \hat{c}_i - \tilde{c}_i| = O(\varepsilon) |c_i| r_i.$$
Hence, we obtain the bound
$$|c_i - \hat{c}_i| \leq \varepsilon |c_i| + O(\varepsilon) |c_i| r_i.$$
Analogously, when $c_i = 0$ we obtain 
$$|c_i - \hat{c}_i| \leq \varepsilon + O(\varepsilon) r_i.$$
It follows that 
\begin{align*}
\norm{(c_0 - \hat{c}_0, \ldots, c_{d-1} - \hat{c}_{d-1}, c_d - c_d)}_2 &\leq \varepsilon \norm{(c_0, \ldots, c_{d})}_2 + O(\varepsilon) \norm{(r_0c_0, \ldots, r_{d}c_{d})}_2\\
&\leq \varepsilon \norm{(c_0, \ldots, c_{d})}_2 + O(\varepsilon) \sqrt{d}\norm{(c_0, \ldots, c_{d})}_2,
\end{align*}
where the last inequality follows from 
\begin{align*}
\max_{i=1,\ldots,d} |c_i| &\leq \norm{(c_0, \ldots, c_{d})}_2 \leq \sqrt{d} \max_{i=1,\ldots,d} |c_i|, \\
\max_{i=1,\ldots,d} |c_i| &\leq \norm{(r_0c_0, \ldots, r_{d}c_{d})}_2 \leq \sqrt{d} \max_{i=1,\ldots,d} |c_i| 
\end{align*}
because $r_i = 1$ for $i = \text{argmax}_{\ell=1, \ldots, d} |c_\ell|$.
\end{proof}
We note that the proof of Proposition \ref{prop:EMBEimpliesNBE} can be summarized as 
$$ \EMBE(\hat{X}) = \varepsilon \overset{\text{Theorem \ref{thm:EMBEimpliesTBE}}}{\Longrightarrow} \TBE(\hat{X}) = O(\varepsilon) \Longrightarrow \NBE(\hat{X}) = O(\varepsilon).$$

\subsection{Tropical Backward Stability implies Element-wise Backward Stability?}\label{subsection3:2}
We now show that a tropical backward error of order $\varepsilon$ also implies a mixed element-wise backward error of the same magnitude under some assumptions.
For this, consider the perturbed polynomials
\begin{align}
\hat{f} = f + \hat{\Delta}f &= \sum_{i=0}^{d-1} c_i(1+r_i \delta_i) x^i + c_d x^d \nonumber\\
\intertext{and}
\tilde{f} = f + \tilde{\Delta}f &= \sum_{i=0}^{d-1} c_i(1+\kappa_i \delta_i e^{\sqrt{-1}\theta_i}) x^i + c_d x^d \label{eq:tildewithparam}
\end{align}
where $\log|\delta_i| \approx \log |\varepsilon| = v_{\varepsilon}$ and $\kappa_i \in \R, \theta_i \in [0,2\pi)$ are parameters. We assume that $\kappa_i$ is not too large, i.e.\ $\log |\kappa_i| \approx 0$, such that $\tilde{\Delta} f$ is a `small' perturbation. 
Observe that for the roots $\hat{x}_j = x_j + \hat{\Delta} x_j$ of $\hat{f}$ we have 
\begin{equation} \label{eq:expansion}
  \begin{array}{rl}
    0 &= (f + \hat{\Delta} f)(x_j + \hat{\Delta} x_j) \\
    &= f(x_j + \hat{\Delta} x_j) + \hat{\Delta} f(x_j + \hat{\Delta} x_j) \\
    &= f(x_j) + f'(x_j) \hat{\Delta} x_j + \frac{f''(x_j)}{2} \hat{\Delta}x_j^2 + \cdots + \frac{f^{(d)}(x_j)}{d!} \hat{\Delta}x_j^d \\ 
    &\phantom{=}+ \hat{\Delta} f(x_j) + \hat{\Delta} f'(x_j) \hat{\Delta} x_j + \frac{\hat{\Delta} f''(x_j)}{2} \hat{\Delta}x_j^2 + \cdots + \frac{\hat{\Delta} f^{(d)}(x_j)}{d!} \hat{\Delta}x_j^d.
  \end{array}
  \end{equation}
From this we conclude that $\hat{\Delta}x_j$ is a root of the polynomial 
$$ \hat{E}(x) = \hat{\Delta}f(x_j) + (f'(x_j) + \hat{\Delta}f'(x_j)) \frac{x}{1!} + \cdots + (f^{(d)}(x_j) + \hat{\Delta}f^{(d)}(x_j)) \frac{x^d}{d!}.$$
Similarly, for the roots $\tilde{x}_1 = x_1+ \tilde{\Delta}x_1, \ldots, \tilde{x}_d= x_d+ \tilde{\Delta}x_d$ of $\tilde{f}$ we have that $\tilde{\Delta}x_j$ is a root of the polynomial 
$$ \tilde{E}(x) = \tilde{\Delta}f(x_j) + (f'(x_j) + \tilde{\Delta}f'(x_j)) \frac{x}{1!} + \cdots + (f^{(d)}(x_j) + \tilde{\Delta}f^{(d)}(x_j))\frac{x^d}{d!}.$$

To show that tropical backward stability implies element-wise mixed backward stability we need three assumptions.
Each tropical root $\tau_i$ should attain the maximum in \eqref{eq:troppol} \emph{exactly} twice, and the other terms should be significantly smaller. Also, the tropical root $\tau_i$ should be of the same order of magnitude as $\log|x_i|$.
\begin{lemma} \label{lem:derivatives}
If for the tropical root $\tau_j$ of $f$ we have 
\begin{enumerate}
  \item $ \{ \beta \in \{0,\ldots, d\} ~|~ v_{\beta} + \beta \tau_j = \max_{i} v_i + i  \tau_j \} = \{\beta_{\ell-1}, \beta_\ell \}$ with $\beta_{\ell-1} < \beta_\ell$,
  \item $\log|x_j| \approx \tau_j$,
\item $|c_i x_j^i| \ll |c_{\beta_\ell} x_j^{\beta_\ell}|, i \neq \beta_{\ell-1},\beta_\ell$,
\end{enumerate}
then for $k \geq 1$
$$\log|f^{(k)}(x_j)| \lesssim v_{\beta_\ell} + (\beta_\ell - k) \tau_j.$$
Here `$\lesssim$' can be replaced by `$\approx$' for $k = 1$.
\end{lemma}
\begin{proof}
We have that $x^k f^{(k)}(x) = \sum_{i = k}^d c_i \frac{i !}{(i-k)!} x^i$. We distinguish three different cases.
\begin{enumerate}
\item ($\beta_{\ell - 1} - k \geq 0, \beta_{\ell} - k \geq 0$). In this case 
$$ |x_j^kf^{(k)}(x_j)| = K_1 \left |c_{\beta_{\ell - 1}} \frac{\beta_{\ell - 1} !}{(\beta_{\ell - 1}-k)!} x_j^{\beta_{\ell - 1}}+ c_{\beta_{\ell}} \frac{\beta_{\ell} !}{(\beta_{\ell}-k)!} x_j^{\beta_{\ell}} \right |,$$
with $\log |K_1| \approx 0$.
Since $f(x_j)= 0$ and by assumption $|c_i x_j^i| \ll |c_{\beta_{\ell}}x_j^{\beta_{\ell}}|, i \neq \beta_{\ell - 1},\beta_{\ell}$, we have that 
\begin{align*}
|c_{\beta_{\ell - 1}}x_j^{\beta_{\ell - 1}} + c_{\beta_{\ell}}x_j^{\beta_{\ell}}| &= \left | \sum_{i \neq \beta_{\ell - 1},\beta_{\ell}} c_i x_j^i \right| 
\leq \sum_{i \neq \beta_{\ell - 1},\beta_{\ell}} |c_i x_j^i| \ll  |c_{\beta_{\ell}}x_j^{\beta_{\ell}}|.
\end{align*}
Then 
\small
\begin{align*}
|x_j^kf^{(k)}(x_j)| &= K_1 \left | \frac{\beta_{\ell - 1} !}{(\beta_{\ell - 1}-k)!} (c_{\beta_{\ell - 1}}x^{\beta_{\ell - 1}} + c_{\beta_{\ell}}x^{\beta_{\ell}}) + \left (\frac{\beta_{\ell} !}{(\beta_{\ell}-k)!} - \frac{\beta_{\ell - 1} !}{(\beta_{\ell - 1}-k)!} \right ) c_{\beta_{\ell}}x_j^{\beta_{\ell}} \right | \\
& = K_2 |c_{\beta_{\ell}}x_j^{\beta_{\ell}}|
\end{align*}
\normalsize
with $\log|K_2| \approx 0$. The lemma now follows from taking valuations.
\item ($\beta_{\ell - 1} - k <0, \beta_{\ell} - k \geq 0$). The lemma follows from the observation that in this case 
 $$ |x_j^kf^{(k)}(x_j)| = K_1 \left | c_{\beta_{\ell}} \frac{\beta_{\ell} !}{(\beta_{\ell}-k)!} x_j^{\beta_{\ell}} \right | = K_2 |c_{\beta_{\ell}}x_j^{\beta_{\ell}}|,$$
 with $\log|K_1| \approx 0, \log|K_2| \approx 0$.
\item($\beta_{\ell - 1} - k < 0, \beta_{\ell} - k < 0$). In this case 
$$ |x_j^kf^{(k)}(x_j)| = \left |\sum_{i = k}^d c_i \frac{i !}{(i-k)!} x_j^i \right | \leq \sum_{i = k}^d \left | c_i \frac{i !}{(i-k)!} x_j^i \right| \ll |c_{\beta_{\ell}}x_j^{\beta_{\ell}}|.$$
\end{enumerate}
Note that if $k = 1$, the third case is not possible because $\beta_{\ell} > \beta_{\ell - 1} \geq 0$.
\end{proof}
\begin{lemma} \label{lem:Ecoeffs}
Under the assumptions of Lemma \ref{lem:derivatives}, we have that 
$$ \log|\hat{\Delta} f(x_j)| \lesssim v_{\beta_{\ell}} + v_{\varepsilon} + \beta_{\ell} \tau_j, \quad  \log|\tilde{\Delta} f(x_j)| \lesssim v_{\beta_{\ell}} + v_{\varepsilon} + \beta_{\ell} \tau_j,$$ 
$$ \log |\hat{\Delta}f'(x_j)| \lesssim v_{\beta_{\ell}}+ v_{\varepsilon} + (\beta_{\ell}-1) \tau_j, \quad \log|\tilde{\Delta}f'(x_j)| \lesssim v_{\beta_{\ell}}+ v_{\varepsilon} + (\beta_{\ell}-1) \tau_j,$$
$$ \log |f^{(k)}(x_j) + \hat{\Delta}f^{(k)}(x_j)| \lesssim v_{\beta_{\ell}} + (\beta_{\ell}-k)\tau_j, \quad \log|f^{(k)}(x_j) + \tilde{\Delta}f^{(k)}(x_j)| \lesssim v_{\beta_{\ell}} + (\beta_{\ell}-k)\tau_j.$$
In the last line, for $k = 1$ we can replace `$\lesssim$' by `$\approx$'.
\end{lemma}
\begin{proof}
We have
\begin{align*}
|\hat{\Delta} f(x_j)| &= \left | \sum_{i=0}^{d-1} c_i \delta_i r_i x_j^i \right|= K_1 \left | \sum_{\beta_{\ell - 1} \leq i \leq \beta_{\ell}}^{d-1} c_i \delta_i r_i x_j^i \right| \leq K_1(\beta_{\ell} - \beta_{\ell - 1})|c_{\beta_{\ell}} \delta_{\beta_{\ell}} x_j^{\beta_{\ell}}|,
\end{align*}
with $\log|K_1(\beta_{\ell} - \beta_{\ell - 1})| \approx 0$, which proves the first statement. The second statement is proven by a completely analogous argument. The third statement follows from 
\begin{align*}
|x_j \hat{\Delta} f'(x_j)| &= \left | \sum_{i=1}^{d-1} c_i i \delta_i r_i x_j^i \right| \leq K_1(\beta_{\ell} - \beta_{\ell - 1})|c_{\beta_{\ell}} \delta_{\beta_{\ell}} x_j^{\beta_{\ell}}|,
\end{align*}
with $\log|K_1(\beta_{\ell} - \beta_{\ell - 1})| \approx 0$. The fourth statement is analogous.
The fifth statement follows from 
$$ \log|f^{(k)}(x_j) + \hat{\Delta}f^{(k)}(x_j)| \approx \log|f^{(k)}(x_j)|$$
and Lemma \ref{lem:derivatives}. The sixth statement follows again from an analogous argument.
\end{proof}
It follows from Lemma \ref{lem:Ecoeffs} that we can bound the lifted Newton polytopes of the polynomials $\hat{E}, \tilde{E}$ from above. An example is shown in Figure \ref{fig:Ecoeffs}.
\begin{figure}
\centering
%
%
\definecolor{mycolor1}{rgb}{0.00000,0.44700,0.74100}%
\definecolor{mycolor2}{rgb}{0.85000,0.32500,0.09800}%
\definecolor{mycolor3}{rgb}{0.92900,0.69400,0.12500}%
\begin{tikzpicture}

\begin{axis}[%
width=4in,
height=3in,
at={(0.758in,0.481in)},
scale only axis,
xmin=-1,
xmax=7,
xtick = \empty,
ymin=-1,
ymax=9,
ytick = \empty,
axis background/.style={fill=white}
]
\addplot [color=mycolor1, line width=1.5pt, mark size=2.5pt, mark=*, mark options={solid, mycolor1}, forget plot]
  table[row sep=crcr]{%
0	0\\
1	7\\
2	6\\
3	5\\
4	4\\
5	3\\
6	2\\
};
\addplot [color=mycolor2, thick, dashed, forget plot]
  table[row sep=crcr]{%
0	8\\
1	7\\
2	6\\
3	5\\
4	4\\
5	3\\
6	2\\
};
\addplot [color=mycolor3, draw=none, mark size=2.5pt, mark=*, mark options={solid, mycolor3}, forget plot]
  table[row sep=crcr]{%
0	8\\
};
\node[xshift = 1.5cm] at (axis cs:0,0) {$v_{\beta_\ell} + v_\epsilon + \beta_\ell \tau_j$};
\node[xshift = 1.3cm] at (axis cs:0,8) {$v_{\beta_\ell} + \beta_\ell \tau_j$};
\node[xshift = 1.7cm] at (axis cs:1,7) {$v_{\beta_\ell} + (\beta_\ell-1) \tau_j$};
\node[xshift = 1.7cm] at (axis cs:2,6) {$v_{\beta_\ell} + (\beta_\ell-2) \tau_j$};
\node[xshift = 0.7cm] at (axis cs:3,5) {$\ldots$};
\node[draw] at (axis cs:1,7) {};
\end{axis}
\end{tikzpicture}%
\caption{The blue line shows an upper bound for the lifted Newton polytopes of $\hat{E}(x)$ and $\tilde{E}(x)$. The actual lifted polytopes will meet the blue line (approximately) in the point $(1, v_{\beta_{\ell}} + (\beta_{\ell} - 1) \tau_j)$ (indicated with a small box). }
\label{fig:Ecoeffs}
\end{figure}
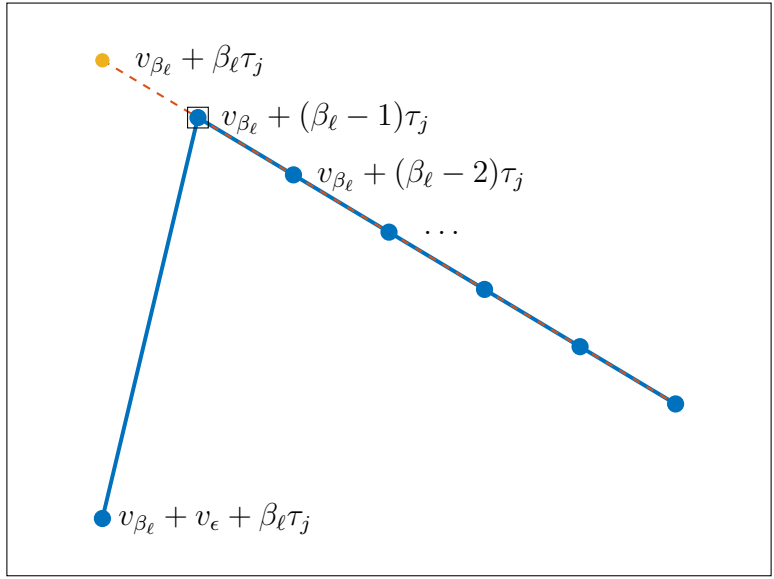
The expansion \eqref{eq:expansion} is used to approximate $\hat{\Delta}x_j$ as
\begin{equation} \label{eq:approxDelta}
\hat{\Delta}x_j \approx \frac{- \hat{\Delta} f (x_j)}{f'(x_j) + \hat{\Delta}f'(x_j)}.
\end{equation} 
It is clear that this is an approximation for the smallest root of $\hat{E}(z)$, which corresponds to the smallest tropical root $\tau_{\hat{E}}$ of $\hat{E}$ which is bounded by (see Figure \ref{fig:Ecoeffs}) 
\begin{equation} \label{eq:tauE}
\tau_{\hat{E}} \leq (v_{\beta_{\ell}} + v_\varepsilon + \beta_{\ell} \tau_j) - (v_{\beta_{\ell}} + (\beta_{\ell} - 1) \tau_j ) = \tau_j + v_{\varepsilon}.
\end{equation}
Analogously, we have for the smallest tropical root $\tau_{\tilde{E}}$ of $\tilde{E}$ that $\tau_{\tilde{E}} \leq \tau_j + v_{\varepsilon}$. We will also make our usual assumption that the tropical roots give an indication for the order of magnitude of the classical roots, i.e.
\begin{equation} \label{eq:stepestim}
\log|\hat{\Delta} x_j | \lesssim \tau_j + v_{\varepsilon}, \quad \log|\tilde{\Delta} x_j | \lesssim \tau_j + v_{\varepsilon}.
\end{equation}
We conclude that the assumptions of Lemma \ref{lem:derivatives} implies that $\hat{X}$ and $\tilde{X}$ have a relative \emph{forward} error of size $O(\varepsilon)$. This implies that $\EMBE(\hat{X}) = O(\varepsilon)$ (take $\tilde{x}_j = x_j$), which gives the following result.
\begin{theorem} \label{thm:TBEimpliesEMBE}
Under the assumptions of Lemma \ref{lem:derivatives}, if $\TBE(\hat{X}) = \varepsilon$ and the order of magnitude of $|x_j|$ is equal to that of $|\hat{x_j}|$ and $\exp(\tau_j)$, then $\EMBE(\hat{X}) = O(\varepsilon)$.
\end{theorem}
In fact, the assumptions of Lemma \ref{lem:derivatives} imply that the roots $X$ of $f$ are well-conditioned. Consider the first order approximation
$$ \frac{|\Delta x_j|}{|x_j|} \approx \frac{\max_{i=0,...,d} |c_i x_j^i|}{|f'(x_j)||x_j|} \frac{|\Delta f(x_j)|}{\max_{i=0,...,d} |c_i x_j^i|}, $$
for a perturbation $\Delta f$ on $f$ causing a perturbation $\Delta x_j$ on $x_j$. Here $\max_{i=0,...,d} |c_i x_j^i|$ is used to measure the residual $|\Delta f(x_j)|$ with a relative criterion. The condition of a root $x_j$ can be measured by 
$$ \frac{\max_{i=0,...,d} |c_i x_j^i|}{|f'(x_j)||x_j|} = \frac{c_{\beta_{\ell}} \exp(\tau_j)^{\beta_\ell}}{|f'(x_j)||x_j|}.$$
Using Lemma \ref{lem:derivatives}, we find that this number is of order 1.

In what follows, we will give a constructive proof for Theorem \ref{thm:TBEimpliesEMBE}. That is, in the proof of Theorem \ref{thm:explicit} (which implies Theorem \ref{thm:TBEimpliesEMBE}) we will give values for $\kappa_i, \theta_i$ in \eqref{eq:tildewithparam} which realize a small EMBE. It uses the following lemma. 
\begin{lemma} \label{lem:epsquad}
Under the assumptions of Lemma \ref{lem:derivatives}, we have that 
$$ \log \left|\hat{\Delta} x_j - \frac{- \hat{\Delta} f (x_j)}{f'(x_j) } \right| \lesssim \tau_j + 2v_\varepsilon, \quad \log \left|\tilde{\Delta} x_j - \frac{- \tilde{\Delta} f (x_j)}{f'(x_j)} \right| \lesssim \tau_j + 2v_\varepsilon. $$
\end{lemma}
\begin{proof}
It follows from $\hat{E}(\hat{\Delta} x_j) = 0$ that 
$$ f'(x_j) \hat{\Delta} x_j = - \left( \hat{\Delta} f(x_j) + \hat{\Delta} x_j  \hat{\Delta}f'(x_j)  + \sum_{k=2}^d \hat{\Delta}x_j^k \frac{f^{(k)}(x_j) + \hat{\Delta}f^{(k)}(x_j)}{d!} \right ).$$
The valuation of the first neglected term in the approximation \eqref{eq:approxDelta} is
$$\log \left |\frac{ \textcolor{mycolor1}{\hat{\Delta}x_j} \textcolor{mycolor2}{ \hat{\Delta}f'(x_j)}}{ \textcolor{mycolor4}{f'(x_j)}} \right | \lesssim \textcolor{mycolor1}{\tau_j + v_\varepsilon} + \textcolor{mycolor2}{v_{\beta_{\ell}} + v_\varepsilon + (\beta_{\ell} - 1)\tau_j} - \textcolor{mycolor4}{(v_{\beta_{\ell}} + (\beta_{\ell} - 1) \tau_j)} = 2v_\varepsilon + \tau_j,$$
where we used Lemma \ref{lem:derivatives}, Lemma \ref{lem:Ecoeffs} and \eqref{eq:stepestim}. For the term corresponding to $k = 2$, we have
$$\log \left |\frac{ \textcolor{mycolor1}{\hat{\Delta}x_j^2} \textcolor{mycolor2}{ f^{(2)}(x_j)}}{2 \textcolor{mycolor4}{f'(x_j)}} \right | \lesssim \textcolor{mycolor1}{2 (\tau_j + v_\varepsilon)} + \textcolor{mycolor2}{v_{\beta_{\ell}} + (\beta_{\ell} - 2)\tau_j} - \textcolor{mycolor4}{(v_{\beta_{\ell}} + (\beta_{\ell} - 1) \tau_j)} = 2v_\varepsilon + \tau_j.$$
For the terms corresponding to higher values of $k$, we get in the same way a valuation of $ \tau_j + k v_\varepsilon < \tau_j + 2 v_\varepsilon$. The reasoning for $\tilde{\Delta} x_j$ is completely analogous.
\end{proof}

\begin{theorem} \label{thm:explicit}
Under the assumptions of Lemma \ref{lem:derivatives}, there are choices of the parameters $\kappa_i$, $\theta_i$ with $\log|\kappa_i| \approx 0$ such that $\log|\hat{x}_j - \tilde{x}_j| \approx v_\varepsilon + \tau_j$.
\end{theorem}
\begin{proof}
By Lemma \ref{lem:epsquad} we have $ \hat{x}_j = x_j - \frac{\hat{\Delta} f(x_j)}{f'(x_j)} + O(\varepsilon^2 \exp(\tau_j))$ and $ \tilde{x}_j = x_j - \frac{\tilde{\Delta} f(x_j)}{f'(x_j)} + O(\varepsilon^2 \exp(\tau_j))$. Hence it suffices to show that the valuation of
\begin{align*}
|\hat{x}_j - \tilde{x}_j| & \approx \left | \frac{1}{f'(x_j)} \right | \left | \sum_{i=0}^{d-1} c_i \delta_i (r_i - \kappa_i e^{\sqrt{-1} \theta_i}) x_j^i \right|
\end{align*} 
is bounded by $v_\varepsilon + \tau_j$. We have
\begin{align*}
\left | \frac{1}{f'(x_j)} \right | \left | \sum_{i=0}^{d-1} c_i \delta_i (r_i - \kappa_i e^{\sqrt{-1} \theta_i}) x_j^i \right| &\leq \left | \frac{1}{f'(x_j)} \right |  \sum_{i=0}^{d-1} |c_i| |\delta_i| |r_i - \kappa_i e^{\sqrt{-1} \theta_i}| |x_j^i |\\
& \leq \varepsilon \left | \frac{1}{f'(x_j)} \right |  \sum_{i=0}^{d-1} |c_i| |r_i - \kappa_i e^{\sqrt{-1} \theta_i}| |x_j^i |.
\end{align*}
We now specify the parameters $\kappa_i$, $\theta_i$. For $r_i = O(1)$, we choose $\kappa_i$, $\theta_i$ such that $r_i = \kappa_i e^{\sqrt{-1} \theta_i}$. Note that $\log|\kappa_i| \approx 0$. For the other $i$, we set $\kappa_i = \theta_i = 0$. We get 
$$|\hat{x}_j - \tilde{x}_j| \leq \varepsilon \left | \frac{1}{f'(x_j)} \right |  \sum_{r_i \gg 1} |c_i| |r_i| |x_j^i |.$$
Since by assumption $\log|x_j| \approx \tau_j$, the dominant terms in the sum are those with $\beta_{\ell - 1} < i < \beta_{\ell}$, where $\tau_{\beta_{\ell - 1}+1} = \cdots = \tau_{\beta_{\ell}} = \tau_j$. Therefore 
$$|\hat{x}_j - \tilde{x}_j| \leq K \varepsilon \left | \frac{1}{f'(x_j)} \right |  \sum_{\beta_{\ell - 1} < i < \beta_{\ell}} |c_i| |r_i| |x_j^i |$$
with $\log|K| \approx 0$. The valuation of one of the terms in the sum is 
$$\log|K| + v_\varepsilon - \log|f'(x_j)| + v_i + v_{\beta_{\ell}} + (\beta_{\ell} - i)\tau_j - v_i + i \tau_j$$
which is equal to $ \log|K| - \log|f'(x_j)| + v_\varepsilon  + v_{\beta_{\ell}} + \beta_{\ell} \tau_j$. Note that this is independent of $i$, and hence we get 
$$\log|\hat{x}_j - \tilde{x}_j| \approx  - \log|f'(x_j)| + v_\varepsilon  + v_{\beta_{\ell}} + \beta_{\ell} \tau_j.$$
Using Lemma \ref{lem:derivatives} we get 
$$\log \left |\frac{\hat{x}_j - \tilde{x}_j}{\tilde{x}_j} \right | \approx  v_\varepsilon  .$$
\end{proof}

\section{Computational Experiments} \label{sec:numexp}

In Subsection \ref{subsection3:2} we proved that a tropical backward error of order $\varepsilon$ also implies a mixed element-wise backward error of the same magnitude under some assumptions.
Unfortunately, we were not able to prove this result in general.
However, based on several numerical experiments that we performed, we are convinced that a small TBE implies a small EMBS also in general.
To support this conjecture, the following numerical experiment was performed.
\\
{\bf Numerical Experiment 1}
Take 1000 polynomials of degree $d$ with coefficients whose modulus is chosen as $10^e$ with $e$ uniformly randomly chosen 
between $-k$ and $k$ and whose argument is uniformly randomly chosen between $0$ and $2 \pi$.
These polynomials will not always satisfy the necessary assumptions for Theorem \ref{thm:TBEimpliesEMBE}.
The zeros of these polynomials are approximated by applying an eigenvalue method from the Julia package \texttt{Polynomials} resulting in the computed zeros $\hat{X}$. For these approximate roots the TBE is computed. To compute an upper bound
for the EMBE, the roots $\hat{x}_j$ are separately refined using Newton's method in extended precision based on the original polynomial.
The correspondence between the TBE and EMBE is shown in Figure~\ref{fig:exp01} and clearly indicates that a small TBE implies a small EMBE.
\begin{figure}[h]
\centering
\includegraphics[width=0.65\textwidth]{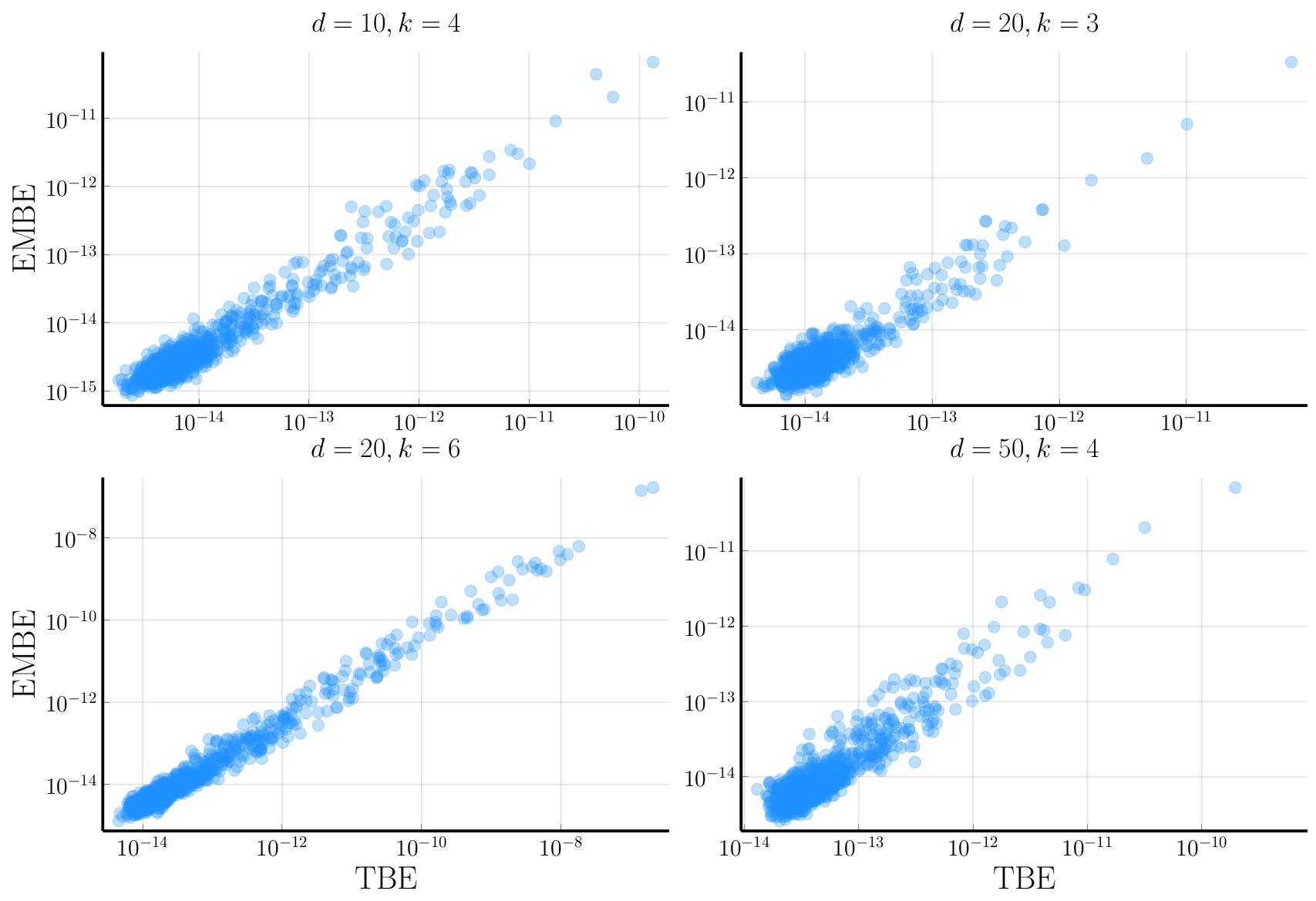}
\caption{Results of Numerical Experiment 1.}
\label{fig:exp01}
\end{figure}

In several of our statements, we assumed that the tropical roots of $f$ are of the same order of magnitude as the corresponding classical roots. To check if this is a reasonable assumption in practice, we performed the following numerical experiment.
\begin{figure}[h]
\centering
\includegraphics[width=0.65\textwidth]{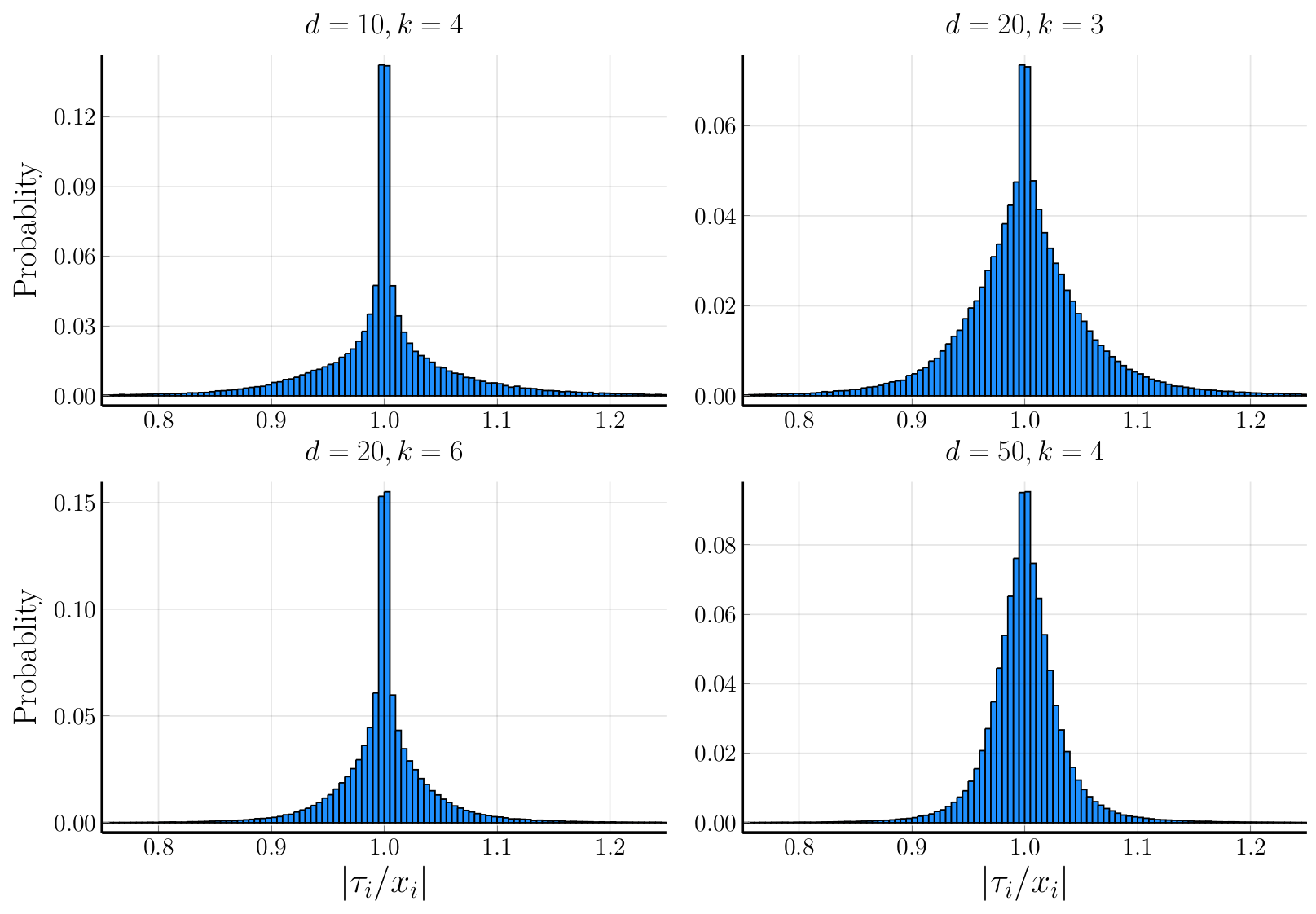}
\caption{Results of Numerical Experiment 2.}
\label{fig:exp02}
\end{figure}
\\
{\bf Numerical Experiment 2}
Ten thousand polynomials of degree $d$ are taken with coefficients whose modulus is chosen as $10^e$ with $e$ uniformly 
randomly chosen 
between $-k$ and $k$ and whose argument is uniformly randomly chosen between $0$ and $2 \pi$.
For each of these polynomials the tropical roots are compared to the roots computed in high precision.
Figure~\ref{fig:exp02} gives a histogram of the measured ratios $|\tau_i/x_i|$.
The results show that for the vast majority of roots the magnitude differs by at most 10 percent.
\section{Conclusion}
We have shown the relations \eqref{eq:diagram} between different measures for the backward error of an approximate set of roots $\hat{X}$ of a polynomial. Under some assumptions the tropical backward error measure of \cite{vanbarel2019tropical}, which is easy to compute, is shown to be equivalent to the element-wise mixed backward error measure defined in \cite{mastronardi2015revisiting} for $d=2$. We have given numerical evidence that the equivalence holds more generally. 

\footnotesize
\bibliography{references}
\bibliographystyle{alpha}

\end{document}